\documentclass{amsart}

\usepackage{amsmath}
\usepackage{amssymb}
\usepackage{amsthm}
\usepackage{amsfonts}
\usepackage{array}
\usepackage[all]{xy}
\usepackage{graphicx}

\newcommand{\Q}{\mathbb{Q}}
\newcommand{\Z}{\mathbb{Z}}
\newcommand{\N}{\mathbb{N}}

\DeclareMathOperator{\im}{im}
\DeclareMathOperator{\kr}{ker}

\newcommand{\isom}{\cong}
\newcommand{\goto}{\rightarrow}

\newcommand{\of}{\circ}

\newcommand{\unkh}{\overline{Kh}}

\newcommand{\Qsub}[1]{\Q_{(#1)}}
\newcommand{\one}{\mathbf{1}}
\newcommand{\eX}{\mathbf{X}}

\theoremstyle{plain}

\newtheorem{theorem}{Theorem}
\newtheorem{conjecture}{Conjecture}

\newtheorem{corollary}[theorem]{Corollary}
\newtheorem{claim}{Claim}

\begin{document}

\title[The Khovanov Homology of $(p,-p,q)$ Pretzel Knots]{The Khovanov Homology of $(p,-p,q)$ Pretzel Knots}

\date{\today}

\author[Starkston]{Laura Starkston}
\maketitle

\begin{abstract}
In this paper, we compute the Khovanov homology over $\Q$ for $(p,-p,q)$ pretzel knots for $3\leq p \leq15$, $p$ odd, and arbitrarily large $q$. We provide a conjecture for the general form of the Khovanov homology  of $(p,-p,q)$ pretzel knots. These computations reveal that these knots have thin Khovanov homology (over $\Q$ or $\Z$). Because Greene has shown that these knots are not quasi-alternating, this provides an infinite class of non-quasi-alternating knots with thin Khovanov homology. 
\end{abstract}

\section{Introduction}
{
In \cite{Khovanov}, Khovanov introduced his categorification of the Jones polynomial, a graded homology theory, which is a powerful link invariant. As with the Jones polynomial, there is a finite way to compute the Khovanov homology of a link given a diagram with finitely many crossings. A number of programs have been written to compute the Khovanov homology of a link. The programs by Bar-Natan and Green can be implemented in Mathematica \cite{KAtlas}. A faster program called KhoHo was written by Shumakovitch \cite{KhoHo}. Of course, one must have a finite description of the knot or link, in order to use these programs to obtain its Khovanov homology. 

We set out to find the Khovanov homology of infinite classes of knots, in particular infinite classes of pretzel knots. Lee proves in \cite{Lee} that the Khovanov homology of alternating links (over rational coefficients) is completely determined by the Jones polynomial and knot signature. Ozsv\'{a}th  and Szab\'{o} defined a larger class of knots they call quasi-alternating links \cite{quasidef} and  Manolescu and Ozsv\'{a}th extended Lee's results to this class in \cite{quasi}. In \cite{Champanerkar} the quasi-alternating status of pretzel links was explored by Champanerkar and Kofman. They classify many pretzel links as either quasi-alternating, or not, leaving open the status of only $2$ classes of $3$ column pretzel links: $P(p,-p,q)$ and $P(p+1,-p,q)$. Greene completes this classification in \cite{Greene}, and finds that the $P(p,-p,q)$ knots are not quasi-alternating, when $q>p$ thus their Khovanov homology is not necessarily determined by their Jones polynomial and signature.

We look at the Khovanov homology of the $P(p,-p,q)$ knots ($q\geq p > 2$, $p$ odd). We utilize the simplicity of the diagrams resulting from resolving one of the crossings in the third column, to make an inductive argument in terms of $q$, which is completed using the fact that these are slice knots. Thus, given a satisfactory base case for some odd value of $p$, we can provide a formula for the Khovanov homology of all $P(p,-p,q)$ knots for $q$ sufficiently high. We show the explicit proof for the $p=3$ case, which applies in exactly the same manner for other values of $p$ once we compute a base case meeting certain criteria. We have already verified that such base cases exist for all odd values $3 \leq p \leq 15$ .

\textbf{Acknowledgements:} Many thanks to Peter Kronheimer for his guidance and advice throughout this project. Thank you also to Joshua Greene for the information on the quasi-alternating status of these knots.
}
\section{Khovanov Homology}
{
 \subsection{Graded Modules}
 {
 Let $A$ be the free graded module generated by two elements, $\one$ and $\eX$ over a ring $R$. We assign a quantum grading to $A$ so that the copy of $R$ generated by $\one$ has quantum grading $1$ and the copy of $R$ generated by $\eX$ has quantum grading $-1$. This induces a quantum grading on $A^{\otimes q}$ where the copy of $R$ generated by $v_1\otimes \cdots \otimes v_q$ has quantum grading equal to the sum of the quantum gradings of $v_1$ through $v_q$.
 
 We will denote quantum gradings by subscripts in parentheses. For example $A \isom R_{(-1)}\oplus R_{(1)}$. Let $\cdot \{k\}$ denote a quantum grading shift up by $k$. So $M_{(q)}\{k\}=M_{(q+k)}$.
 
 We can turn $A$ into a bialgebra by defining a multiplication, $m$, a comultiplication, $\Delta$, a unit, and a counit. We will only be concerned with $m$ and $\Delta$ here but Khovanov defines the others in \cite{Khovanov}.
 
   \begin{eqnarray*}
  m(\one \otimes \one) & = & \one.\\
  m(\one \otimes \eX) & = & \eX.\\
  m(\eX \otimes \one) & = & \eX.\\
  m(\eX \otimes \eX) & = & 0.
  \end{eqnarray*}
  \begin{eqnarray*}
  \Delta(\one) & = & \one \otimes \eX + \eX \otimes \one.\\
  \Delta(\eX) & = & \eX \otimes \eX.
  \end{eqnarray*}

 }
 \subsection{$n$-cube of smoothings}
 {
  For any given crossing there are two ways to resolve the crossing to eliminate it. We call these the $0$-smoothing and the $1$-smoothing according to the convention in Figure \ref{smoothings}.
  \begin{figure}[here]
  \includegraphics[scale=1.2]{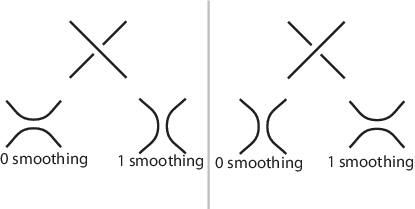}
  \caption{The $0$ and $1$ resolutions of a crossing}
  \label{smoothings}
  \end{figure}
  
  Given a knot or link diagram $D$, with $n$ crossings, there are $2^n$ different total resolutions of the diagram. Each total resolution is a collection of simple closed curves. Additionally, each resolution corresponds to an $n$-tuple of $0$s and $1$s. Fixing an ordering of the crossings, thus gives an identification of the collection of total resolutions of $D$ with the vertices of the unit $n$-cube. See Figure \ref{trefoilcube} for an example.
 
 \begin{figure}[here]
 \includegraphics[scale=2]{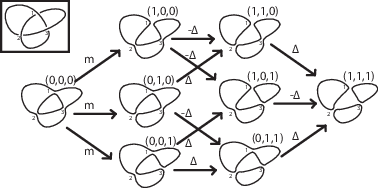}
 \caption{The cube of total smoothings of the trefoil knot}
 \label{trefoilcube}
 \end{figure}
  
  To each vertex $v=(x_1,\cdots, x_n) \in \{0,1\}^n$, the \emph{weight} of the corresponding smoothing, $D_v$ is given by $w(D_v)=\sum_{i=1}^nx_i$ (i.e. the number of $1$-smoothings). Each of the edges of the $n$-cube corresponds to a change of one crossing from a $0$-smoothing to a $1$-smoothing. We label this edge with a map. The map is labeled $m$ if two circles in the $0$-smoothing merge to one circle in the $1$-smoothing, and $\Delta$ if one circle in the $0$-smoothing divides into two in the $1$-smoothing. We add in signs so that each square anti-commutes.
  
  We then translate the diagram of smoothings into an algebraic diagram of free modules with a functor $\mathcal{F}$. Each complete smoothing $D_v$ translates to $\mathcal{F}(D_v)=A^{\otimes k}\{w(D_v)\}$ where $k$ is the number of closed curves in the smoothing. The maps $m$ and $\Delta$ translate to the multiplication and comultiplication on the copies of $A$ corresponding to the circles that are merging or dividing.
  
 The resulting diagram for the trefoil is
 $$ \xymatrix
 {
  & A\{1\} \ar[r]^{-\Delta} \ar[dr]^{-\Delta\qquad\qquad\qquad\qquad\qquad\qquad\qquad} & A\otimes A \{2\} \ar[dr]^{\Delta} & \\
 A\otimes A \ar[ur]^{m} \ar[r]^{m} \ar[dr]^{m} & A\{1\} \ar[ur]_{\Delta\qquad\qquad\qquad\qquad\qquad\qquad} \ar[dr]^{-\Delta \qquad\qquad\qquad\qquad\qquad\qquad\qquad} & A\otimes A \{ 2\} \ar[r]^{-\Delta} & A\otimes A \otimes A \{3\}\\
  & A\{1\} \ar[r]^{\Delta} \ar[ur]^{\qquad\qquad\qquad\qquad\qquad\Delta} & A\otimes A \{2\} \ar[ur]_{\Delta} & 
 .}$$
 }
 \subsection{Khovanov Complex}
 {
 From the $n$-cube of modules, we obtain a chain complex for the diagram, $\overline{CKh}(D)$ in the following manner. 
 $$\overline{CKh}^i(D)=\bigoplus_{\{D_v \colon w(D_v)=i\}}\mathcal{F}(D_v).$$
 In other words, if we align the vertices of the $n$-cube such that vertices of the same weight are in the same column as above, we simply take the direct sum down each column. The differentials are given by the sums of the maps in the cube from smoothings of weight $i$ to smoothings of weight $i+1$. Because each square anticommutes, $d\of d=0$.
 
 We let $\cdot [k]$ denote a shift up by $k$ in the homological grading ($\overline{CKh}^i(D)[k]=\overline{CKh}^{i+k}(D)$). Let $n_+$ be the number of $(+)$ crossings and $n_-$ be the number of $(-)$ crossings according to the convention in figure \ref{posneg}.
 
 \begin{figure}[here]
 \includegraphics[scale=.7]{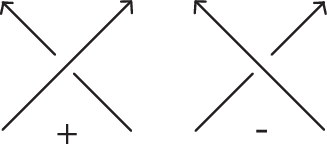}
 \caption{Positive and negative crossings}
 \label{posneg}
 \end{figure}
 
 Let $CKh(D)=\overline{CKh}(D)[-n_-]\{n_+-2n_-\}$. Let $\overline{Kh}(D)$ and $Kh(D)$ be the cohomologies of $\overline{CKh}(D)$ and $CKh(D)$ respectively. $Kh(D)$ is Khovanov's link invariant. $\overline{Kh}(D)$ is specific to the particular diagram in the sense that it depends on the number of $+$ and $-$ crossings in the diagram.
 }
 
\subsection{A Long Exact Sequence}
{
Suppose we start with a diagram $D$, and resolve only one crossing. Let the diagrams with the $0$-resolution and $1$-resolution of this crossing be denoted $D(*0)$ and $D(*1)$ respectively. It is clear from the construction that $\overline{CKh}(D(*0))$ and $\overline{CKh}(D(*1))$ are disjoint subcomplexes of $\overline{CKh}(D)$, except that $\overline{CKh}(D(*1))$ has its homological and quantum gradings shifted up by $1$ in $\overline{CKh}(D)$. Additionally their union consists of all the vertices of $\overline{CKh}(D)$ and $\overline{CKh}(D)$ is the total complex of $\overline{CKh}(D(*0))\goto \overline{CKh}(D(*1))$. Therefore we get a short exact sequence of complexes:
$$0 \goto \overline{CKh}(D(*1))\{1\}[1] \goto \overline{CKh}(D) \goto \overline{CKh}(D(*0)) \goto 0.$$
This induces a long exact sequence on the homology:
\begin{equation}
\cdots \goto \overline{Kh}^{n-1}(D(*1))\{1\} \goto \overline{Kh}^{n}(D) \goto \overline{Kh}^n(D(*0)) \goto \overline{Kh}^n(D(*1))\{1\} \goto \cdots.
\label{LES}
\end{equation}
}
}
\section{Lee's variant of Khovanov Homology and Rasmussen's s-invariant}
{
\subsection{Lee's Invariant}
{
Lee considered a similar construction in \cite{Lee}. Her complex comes from the same $n$-cube of smoothings, and the modules associated to the vertices are the same as Khovanov's, although Lee takes the coefficient ring to be the rational numbers. The main difference is that the maps $m$ and $\Delta$ are slightly modified. In Lee's construction she uses the following maps as the multiplication and comultiplication on $A$.
   \begin{eqnarray*}
  m'(\one \otimes \one) & = & \one.\\
  m'(\one \otimes \eX) & = & \eX.\\
  m'(\eX \otimes \one) & = & \eX.\\
  m'(\eX \otimes \eX) & = & \one.
  \end{eqnarray*}
  \begin{eqnarray*}
  \Delta'(\one) & = & \one \otimes \eX + \eX \otimes \one.\\
  \Delta'(\eX) & = & \eX \otimes \eX + \one \otimes \one.
  \end{eqnarray*}
  Let $CKh'(L)$ be the analogous complex to that constructed in section $2$, but replacing $m$ and $\Delta$ by $m'$ and $\Delta'$ respectively. With the grading shifts included, Khovanov's original maps $m$ and $\Delta$ were constructed so that the differentials preserve quantum grading. (Each map $m$ or $\Delta$ decreases quantum grading by $1$, but as the weight increases by $1$, there is a quantum grading shift by $1$ which cancels out the decrease in quantum grading by the $m$ or $\Delta$ map.)
 
 However, Lee's differentials do not preserve quantum grading, and $\Delta(\eX)$ does not even have a homogeneous quantum grading. However each monomial in the image of some monomial $x$ under a differential map has quantum grading greater than or equal to the quantum grading of $x$. Let $q(x)$ denote the quantum grading of $x$. Then the differential respects the following filtration:
  $$F^p CKh'(L) = \{x\in CKh'(L) \colon q(x) \geq p\}.$$
This filtration together with the homological grading induces a spectral sequence whose $E_0$ term is $CKh'(L)\isom CKh(L)$. The differential on the $0^{th}$ page, $d_0\colon E_0^{q,r}\goto E_0^{q,r+1}$ is the part of Lee's differential that preserves quantum grading. This is exactly Khovanov's original differential, so the $E_1$ page is given by $E_1^{q,r}=Kh^r(K)_{(q)}$.

\textbf{Important Note:} While it is typical in writing the spectral sequence for a filtered differential graded module $A$  to use notation such that $E_1^{q,r} \isom H^{r+q}(F^rA/F^{r+1}A)$ and the $r^{th}$ differential has bidegree $(r, 1-r)$ it is more natural in this context to let $E_1^{r,q}=F^qKh^r(L)/F^{q+1}Kh^r(L) \isom Kh^r(K)_{(q)}$. Using this indexing, the $r^{th}$ differential has bidegree $(1,r)$ (homological degree increases by $1$ and filtration degree increases by $r$).

Lee proves that the rank of $Kh'(L)$ is simply determined by the number of components of the link, $L$. If $L$ has $n$ components the rank of $Kh'(L)$ is $2^n$. Therefore if $K$ is a knot, $Kh'(K)$ has rank $2$.
}
\subsection{Rasmussen's $s$-invariant}
In \cite{Rasmussen} Rasmussen asks, what are the quantum gradings of these two remaining copies of $\Q$ in the $E_{\infty}$ page of the spectral sequence described above? He proves that the difference between the two quantum gradings is exactly $2$ and then defines an invariant $s(K)$ to be the average of these two quantum gradings. Furthermore he proves that $s(K)$ provides a lower bound on the slice genus of the knot:
$$|s(K)| \leq 2g^*(K).$$
}

\section{Main Result}
We compute the Khovanov homology for the class of $(3,-3,q)$ pretzel knots. This same proof can be used to compute $P(p,-p,q)$ knots once an appropriate base case for the induction can be found.
{\begin{theorem}
Let $K_q=P(3,-3,q)$, the 3-stranded pretzel knot where $q \geq 5$. Then
\begin{eqnarray*}
Kh^0(K_q) & = & \Qsub{-1} \oplus \Qsub{1}\\
Kh^i(K_q) & = & 0 \;\;\;\; (0 < i \leq q-4)\\
Kh^{q-3}(K_q) & = & \Qsub{1+2(q-4)}\\
Kh^{q-2}(K_q) & = & \Qsub{5+2(q-4)}\\
Kh^{q-1}(K_q) & = & \Qsub{5+2(q-4)}\\
Kh^{q}(K_q)& = & \Qsub{7 + 2(q-4)} \oplus \Qsub{9 + 2(q-4)}\\
Kh^{q+1}(K_q) & = & \Qsub{11+2(q-4)}\\
Kh^{q+2}(K_q) & = & \Qsub{11+2(q-4)}\\
Kh^{q+3}(K_q) & = & \Qsub{15+2(q-4)}
\end{eqnarray*}
and $Kh^j(K_q)=0$ for all other values of $j$.
\label{mainthm}
\end{theorem}
\begin{proof}
We proceed by induction on $q$. The case for $q=5$ can be verified computationally \cite{KAtlas}.

For $q>5$ we consider the following 3 diagrams.
\begin{center}
\includegraphics[scale = 1]{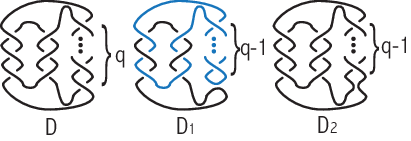}
\end{center}

Notice that $D_1$ is the diagram $D$ with the last crossing in the last column resolved in the $0$-smoothing. Likewise $D_2$ is $D$ with the crossing resolved in the $1$-smoothing. Therefore the long exact sequence (\ref{LES}) gives
$$\cdots \goto \unkh^{n-1}(D_2)\{1\} \goto \unkh^n(D) \goto \unkh^n(D_1) \goto \unkh^n(D_2)\{1\} \goto \cdots.$$
We observe that $D_1$ is a diagram for the $2$-component unlink. Therefore 
\begin{eqnarray*}
Kh^0(D_1) &=& \Qsub{-2} \oplus \Qsub{0}^2 \oplus \Qsub{2}.\\
Kh^i(D_1) & = & 0 \;\;\;\; (\forall i \neq 0).
\end{eqnarray*}
Whichever orientation we put on $D_1$, $n_+ = 3+q-1$ and $n_- = 3$. Therefore
\begin{eqnarray*}
\unkh^3(D_1) & = & \Qsub{2-q} \oplus \Qsub{4-q}^2 \oplus \Qsub{6-q}.\\
\unkh^i(D_1) & = & 0 \;\;\;\; (\forall i \neq 3).
\end{eqnarray*}
Inserting these values into the long exact sequence (\ref{LES}) we get the following
$$0 \goto \unkh^{i-1}(D_2)\{1\} \goto \unkh^i(D) \goto 0$$
for all $i\neq 3, 4$. This implies that $\unkh^{i-1}(D_2)\{1\} \isom \unkh^i(D)$ for $i\neq 3,4$.

Noticing that $D_2$ is a diagram for the $P(3,-3,q-1)$ pretzel knot, we can use the inductive hypothesis to get its Khovanov Homology. Noticing that for $D_2$, $n_+ = 3+q-1$ and $n_- = 3$ the unnormalized Khovanov homology for $D_2$ is given by the column on the left below. The right hand column shifts the gradings by $1$, to give the appropriate isomorphisms. 
$$\begin{array}{rclcrcl}
\unkh^3(D_2) & = & \Qsub{3-q} \oplus \Qsub{5-q} & \overset{\{1\}}{\goto}& \Qsub{4-q} \oplus \Qsub{6-q} &  &\\
\unkh^{3+i}(D_2) & = & 0 \quad (0 < i \leq q-5)&& 0 & \isom & \unkh^{4+i}(D) \quad (1 < i \leq q-5)\\
\unkh^{q-1}(D_2) & = & \Qsub{1+q-5} &\overset{\{1\}}{\goto}& \Qsub{1+q-4}& \isom & \unkh^{q}(D)\\
\unkh^{q}(D_2) & = & \Qsub{5+q-5}&& \Qsub{5+q-4} & \isom & \unkh^{q+1}(D)\\
\unkh^{q+1}(D_2) & = & \Qsub{5+q-5} &\overset{\{1\}}{\goto}& \Qsub{5+q-4} & \isom & \unkh^{q+2}(D)\\
\unkh^{q+2}(D_2) & = & \Qsub{7 + q-5} \oplus \Qsub{9+q-5} && \Qsub{7+q-4}\oplus \Qsub{9+q-4} & \isom & \unkh^{q+3}(D)\\
\unkh^{q+3}(D_2) & = & \Qsub{11+q-5} &\overset{\{1\}}{\goto}& \Qsub{11+q-4} & \isom & \unkh^{q+4}(D)\\
\unkh^{q+4}(D_2) & = & \Qsub{11+q-5} && \Qsub{11+q-4} & \isom & \unkh^{q+5}(D)\\
\unkh^{q+5}(D_2) & = & \Qsub{15+q-5} &\overset{\{1\}}{\goto}& \Qsub{15+q-4} & \isom & \unkh^{q+6}(D)
\end{array}$$
After we shift the gradings back to get $Kh(D)$, we find that we have proven the result for $Kh^i(K_q)$ for all $i\neq 0,1$. We are left to find $\unkh^3(D)$ and $\unkh^4(D)$ (which normalize to $Kh^0(K_q)$ and $Kh^1(K_q)$ under the grading shifts). We have the following exact sequence from the long exact sequence

\begin{equation}
\label{SES}
0 \goto \unkh^3(D) \overset{\alpha}{\goto} \Qsub{2-q} \oplus \Qsub{4-q}^2 \oplus \Qsub{6-q} \overset{\beta}{\goto} \Qsub{4-q}\oplus \Qsub{6-q} \overset{\gamma}{\goto} \unkh^4(D) \goto 0.
\end{equation}
Since $\alpha, \beta,$ and $\gamma$ preserve quantum grading, basic linear algebra implies that
$$\Qsub{2-q}\oplus \Qsub{4-q} \subseteq \kr(\beta) = \im(\alpha).$$
Since $\alpha$ is injective, $\unkh^3(D) \isom \im(\alpha)$ therefore
$$\Qsub{2-q} \oplus \Qsub{4-q} \subseteq \unkh^3(D).$$
This results in $4$ possibilities for the isomorphism class of $\unkh^3(D)$. Using the exact sequence (\ref{SES}), we can determine the isomorphism class of $\unkh^4(D)$ corresponding to each of these four possibilities:
$$
\begin{array}{rclrcl}
\unkh^3(D) & = & \Qsub{2-q}\oplus \Qsub{4-q} & \unkh^4(D) & = & 0,\\
\unkh^3(D) & = & \Qsub{2-q}\oplus \Qsub{4-q}^2 & \unkh^4(D) & = & \Qsub{4-q}, \\
\unkh^3(D) & = & \Qsub{2-q}\oplus \Qsub{4-q} \oplus \Qsub{6-q} & \unkh^4(D) & = & \Qsub{6-q},\\
\unkh^3(D) & = & \Qsub{2-q}\oplus \Qsub{4-q}^2 \oplus \Qsub{6-q} & \unkh^4(D) & = & \Qsub{4-q} \oplus \Qsub{6-q}.
\end{array}
$$
After normalization these four possibilities are:
$$
\begin{array}{rclrcl}
Kh^0(D) & = & \Qsub{-1}\oplus \Qsub{1} & Kh^1(D) & = & 0,\\
Kh^0(D) & = & \Qsub{-1}\oplus \Qsub{1}^2 & Kh^1(D) & = & \Qsub{1}, \\
Kh^0(D) & = & \Qsub{-1}\oplus \Qsub{1} \oplus \Qsub{3} & Kh^1(D) & = & \Qsub{3},\\
Kh^0(D) & = & \Qsub{-1}\oplus \Qsub{1}^2 \oplus \Qsub{3} &Kh^1(D) & = & \Qsub{1} \oplus \Qsub{3}.
\end{array}
$$
We aim to show that the first of these four possibilities is correct.

We now utilize the results of Lee and Rasmussen, described in section 3. In particular we recall Rasmussen's invariant, $s(K)$ and its relation to the slice genus.
$$|s(K)| \leq 2g^*(K).$$
\begin{claim}
$g^*(K_q) = 0$ and thus $s(K_q)=0$.
\end{claim}
We observe that there is a cobordism between $K_q$ and the two component unlink. We may resolve one crossing of $K_q$ to obtain a diagram of a link in the isotopy class of the two component unlink. Resolving the crossing amounts to adding in a band splitting the knot into two components. (See Figure \ref{cobordism}).

\begin{figure}[here]
\includegraphics[scale = 1]{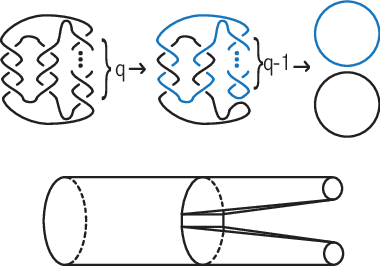}
\caption{Cobordism between $K_q$ and the 2-component unlink}
\label{cobordism}
\end{figure}

We may then cap off the unlink end with 2 discs to obtain a disc bounding $K_q$. Therefore $K_q$ is a slice knot ($g^*(K_q)=0$). Therefore $|s(K_q)| \leq 0$ so $s(K_q)=0$.

We now complete the proof. Because $s(K_q) = 0$, the remaining copies of $\Q$ in the $E_{\infty}$ page of Lee's spectral sequence must be in quantum gradings $-1$ and $1$. The only copies of $\Q$ in those gradings in the $E_1$ term are in the zeroeth homological grading (Table \ref{spectral}). Therefore all differentials on these two copies of $\Q$ must be trivial so that they survive to the $E_{\infty}$ page. All other copies of $\Q$ on the $E_1$ page must not survive to the $E_{\infty}$ page.
\begin{table}[here]
$$
\begin{array}{|c|c|c|c|c|c|c|c|c|c|c|c|c|}
\hline
q+3& & & & & & & & & & & & \Q\\
\hline
 q+2& & & & & & & & & & \Q & & \\
\hline
q+1& & & & & & & & & & \Q & & \\
\hline
q& & & & & & & & \Q & \Q & & & \\
\hline
q-1& & & & & & & \Q & & & & & \\
\hline
q-2& & & & & & & \Q & & & & & \\
\hline
q-3& & & & & \Q & & & & & & & \\
\hline
\vdots& & & & \vdots & \vdots & & & & & & & \\
\hline
2& & &  & & & & & & & & & \\
\hline
1& & \Q^a& \Q^b& & & & & & & & & \\
\hline
0& \Q & \Q^{1+a} & \Q^b & & & & & & & & & \\
\hline
 & \;\; -1 \;\;& \;\;\; 1 \;\;\; & \;\;\; 3 \;\;\;& \cdots & 1+ \tau & 3+\tau & 5+\tau & 7+\tau & 9+\tau & 11+\tau & 13+\tau & 15+\tau \\
\hline
\end{array}
$$
\caption{The $E_1$ page of the Spectral Sequence converging to $Kh'(K_q)$. Note that all empty boxes and boxes that are not shown are trivial. $\tau=2(q-4)$. The vertical axis corresponds to the homological grading in Khovanov homology while the horizontal axis corresponds to the filtration which corresponds to the quantum grading on the $E_1$ page.}
\label{spectral}
\end{table}

Recall that the $r^{th}$ differential goes up $1$ and over $r$, because of an indexing that differs from the standard indexing for a spectral sequence induced by a filtration. (See the note in section 3.1 for further explanation). Let $d_r^{p,q}$ denote the differential on the $r^{th}$ page from $E_r^{p,q}$ to $E_r^{p+1,q+r}$. We know that $d_r^{0,-1} = 0$ and $d_r^{0,1}$ acts trivially on one copy of $\Q$ for every $r$ based on the value of $s(K_q)$. Because the row corresponding to the second homological grading has only zeros, $d_r^{1,1}=0$, for all $r\geq 1$. Thus if $a\neq 0$, an additional copy of $\Q$ will survive in $E_{\infty}^{1,1}$, contradicting Lee's result that there can only be two copies of $\Q$ on the $E_{\infty}$ page. Therefore $a=0$ and $d_r^{0,1}=0$ for all $r\geq 1$. Because the row corresponding to the first homological grading has zeros in quantum gradings greater than $3$, $d_r^{0,3}=0$ for all $r\geq 1$. Therefore if $b\neq 0$, an additional copy of $\Q$ will survive in $E_{\infty}^{0,3}$, again contradicting Lee's result. Therefore $a=b=0$, and the Khovanov homology of $K_q$ is as stated in the theorem.
\end{proof}
}

\section{Coefficients in $\Z$}
{
While Lee's spectral sequence only applies to Khovanov homology with coefficients in $\Q$, the long exact sequence applies with arbitrary coefficients. We can compute a base case over $\Z$ using \cite{KAtlas} or \cite{KhoHo}. Then we obtain isomorphisms as in the previous section  $\unkh^{i-1}(P(3,-3,q-1))\{1\} \isom \unkh^i(P(p,-p,q))$ for $i\neq 3,4$ which determines the values of $Kh(P(3,-3,q)$ in homological degree $r\neq 0,1$. The values in homological degree $r\geq q-3$ or $r<0$ are completely determined by the base case. We note that support of the Khovanov homology in these homological degrees lies in the two main diagonals. The values in homological degree $1<r<q-3$ are determined by $Kh^1(P(3,-3,q-r+1))$. Given our results over $\Q$ combined with the ``interesting'' part of the long exact sequence, we find that $Kh^0$ has a copy of $\Z$ in each of the $-1^{st}$ and $1^{st}$ gradings and $Kh^i$ is either only $0$ or has some torsion contained in the two main diagonals. We illustrate this process in tables \ref{Khov336} and \ref{Khov337}, indicating where the inductive step allows for one torsion factor with the marking $T_i, T_i'$. Here, $i\in \N$ indicates the value of $q$ for which that torsion factor first appeared in the Khovanov homology of the $(p,-p,q)$ pretzel knot. In fact $T_i, T_i'=0$ in these particular cases (verifiable by computation). 

\begin{table}[here]
$$
\begin{array}{|c|c|c|c|c|c|c|c|c|c|c|c|}
\hline
9&&&&&&&&&&&\Z\\\hline
8&&&&&&&&&\Z&\Z_2&\\\hline
7&&&&&&&&&\Z&&\\\hline
6&&&&&&&\Z&\Z\oplus\Z_2&&&\\\hline
5&&&&&&\Z&\Z_2&&&&\\\hline
4&&&&&&\Z&&&&&\\\hline
3&&&&\Z&\Z_2&&&&&&\\\hline
2&&&&&&&&&&&\\\hline
1&&T_{6}&T_{6}'&&&&&&&&\\\hline
0&\Z&\Z&&&&&&&&&\\\hline
&-1&1&3&5&7&9&11&13&15&17&19\\\hline
\end{array}
$$
\caption{The Khovanov homology over $\Z$ of $P(3,-3,6)$. The homological grading is on the vertical axis, and the quantum grading on the horizontal axis to display the similarity with the spectral sequence page above. The unknown pieces are $T_i,T_i'$ as described above.}
\label{Khov336}

$$
\begin{array}{|c|c|c|c|c|c|c|c|c|c|c|c|c|}
\hline
10&&&&&&&&&&&&\Z\\\hline
9&&&&&&&&&&\Z&\Z_2&\\\hline
8&&&&&&&&&&\Z&&\\\hline
7&&&&&&&&\Z&\Z\oplus\Z_2&&&\\\hline
6&&&&&&&\Z&\Z_2&&&&\\\hline
5&&&&&&&\Z&&&&&\\\hline
4&&&&&\Z&\Z_2&&&&&&\\\hline
3&&&&&&&&&&&&\\\hline
2&&&T_{6}&T_{6}'&&&&&&&&\\\hline
1&&T_{7}&T_{7}'&&&&&&&&&\\\hline
0&\Z&\Z&&&&&&&&&&\\\hline
&-1&1&3&5&7&9&11&13&15&17&19&21\\\hline
\end{array}
$$
\caption{The Khovanov homology over $\Z$ of $P(3,-3,7)$. The homological grading is on the vertical axis, and the quantum grading on the horizontal axis to display the similarity with the spectral sequence page above. The unknown pieces are $T_i,T_i'$ as described above.}
\label{Khov337}
\end{table}

Based on computations of the cases for small values of $q$, one expects that these torsion factors will all be zero.
\begin{conjecture}
There are no additional torsion factors appearing in the Khovanov homology for $P(p,-p,q)$ for $q>p+2$. In the notation of tables \ref{Khov336} and \ref{Khov337}, $T_i=T_i'=0$ for all $i$.
\end{conjecture}

Note that the rest of the Khovanov homology is fully determined. There are only two places where the long exact sequence allows new torsion to occur at each stage: in the first homological grading at the first and third quantum gradings and this torsion will persist only within the two diagonals, so the homology over $\Z$ coefficients remains thin.

\begin{corollary}
The $(3,-3,q)$ pretzel knots for $q\geq 4$ have thin Khovanov homology over $\Z$.
\end{corollary}

Combining this with Greene's result that these knots are not quasi-alternating \cite{Greene}, we find that we have an infinite class of non-quasi-alternating knots with thin Khovanov homology. We generalize this result in the next section to other odd values of $p$. Note that for even values of $p$, the Khovanov homology of the link can have torsion off the main diagonals and thus is not always thin.
}

\section{$(p,-p,q)$ Pretzel Knots}
{
We can make the same argument to compute the Khovanov homology of other $(p,-p,q)$ pretzel knots, providing we have a base case knot $B_p$ that satisfies the following conditions
\begin{enumerate}
\item $Kh^0(B_p)=\Qsub{-1}\oplus \Qsub{1}$
\item $Kh^{1}(B_p)=0$
\item  $Kh^i(B_p)=0$ for $i<0$.
\end{enumerate}
(Alternatively we could have $Kh^{-1}=0$ and $Kh^i=0$ for $i>0$ which is the case for the mirror images of these knots).

The other nontrivial groups will determine the formula for the Khovanov homology of the $P(p,-p,q)$ knots. Namely, all $P(p,-p,q)$ knots with $q$ greater than or equal to the value of $q$ in the base case, will have $Kh^0=\Qsub{-1}\oplus \Qsub{1}$, then $Kh^i=0$ for $1\leq i \leq q+c$ where $c$ is some constant determined by the base case. The subsequent groups will be shifted versions of the higher nontrivial groups in the base case.

Homological thinness over $\Z$ coefficients will also follow in the same way as for $p=3$ from homological thinness of the base case over $\Z$, though additional torsion could theoretically show up within the main diagonals as in the previous case.

We have computationally verified that we have base cases satisfying the three required conditions for odd values $3\leq p \leq 15$. In particular the base case for these values of $p$ always occurs in the $P(p,-p,p+2)$ knot. (Recall that in the case $p=3$, the base case is $P(3,-3,5)$.) We suspect that in general the knot $P(p,-p,p+2)$ will provide the appropriate base case. Furthermore, after examining the first few cases, a clear pattern seems to arise. We can therefore extend the result slightly

\begin{theorem}
For $p=3,5,7,9,11,13,15$ and any $q\geq p+2$, the Khovanov homology for $P(p,-p,q)$ over rational coefficients is
\begin{eqnarray*}
Kh^0 & = & \Qsub{-1} \oplus \Qsub{1}\\
Kh^i & = & 0 \;\;\; (0<i\leq q-p-1)\\
Kh^{q-p} & = & \Qsub{3+2(q-p-2)} \\
Kh^{q-p+1} & = & \Qsub{7+2(q-p-2)} \\
Kh^{q-p-2+2i} & = & \Qsub{4i-1+2(q-p-2)}^{i} \oplus \Qsub{4i+1+2(q-p-2)}^{i-2} \;\;\;\; (1 < i \leq n)\\
Kh^{q-p-2+2i+1} & = & \Qsub{4i+1+2(q-p-2)}^{i-1} \oplus \Qsub{4i+3+2(q-p-2)}^i \;\;\;\; (1 < i \leq n)\\
Kh^{q-1} & = & \Qsub{2p+1+2(q-p-2)}^n \oplus \Qsub{2p+3+2(q-p-2)}^{n-1}\\
Kh^{q} & = & \Qsub{2p+3+2(q-p-2)}^n \oplus \Qsub{2p+5+2(q-p-2)}^n\\
Kh^{q-p+1} & = & \Qsub{2p+5+2(q-p-2)}^{n-1} \oplus \Qsub{2p+7+2(q-p-2)}^n\\
Kh^{q+p+1-2i} & = & \Qsub{4p+5-4i+2(q-p-2)}^i \oplus \Qsub{4p+7-4i+2(q-p-2)}^{i-1} \;\;\;\; (1<i\leq n)\\
Kh^{q+p+2-2i} & = & \Qsub{4p+7-4i+2(q-p-2)}^{i-2} \oplus \Qsub{4p+9-4i+2(q-p-2)}^i \;\;\;\; (1<i\leq n)\\
Kh^{q+p-1} & = & \Qsub{4p+1+2(q-p-2)}\\
Kh^{q+p} & = & \Qsub{4p+5+2(q-p-2)}
\end{eqnarray*}
where $n=(p-1)/2$.
\label{ppqthm}
\end{theorem}

\begin{conjecture}
The formula in theorem \ref{ppqthm} holds for $P(p,-p,q)$ for all odd values of $p\geq 17$ and $q\geq p+2$.
\end{conjecture}

When $p$ is even, $P(p,-p,q)$ is a two or three component link (depending on the parity of $q$) and the computation is considerably more complicated. There can be nontrivial homology outside of the main diagonals when $p$ is even so the Khovanov homology is not thin. For example, the Khovanov homology over $\Z$ of $P(2,-2,5)=L_{9n4}$, ($9^2_{43}$ in Rolfsen notation) has two copies of $\Z$ off the main diagonal. Because we can work easily only with odd values of $p$, establishing the base cases via induction on $p$ would require a way to jump from the computation of $P(p,-p,p+2)$ to $P(p+2,-p-2, p+4)$, where the total number of crossings increases by $6$. Thus a different method of induction than the long exact sequence used for theorem \ref{mainthm} would be required for the full generalization to arbitrary odd values of $p$.

}

%
%

\bibliographystyle{amsplain}

\end{document}